\documentclass[11pt]{amsart}
\pdfoutput=1

\usepackage{mathptmx}
\usepackage{amssymb, amsfonts, amsthm}
\usepackage{tikz}
\usepackage{pgf}
\usepackage{epstopdf}
\usepackage{booktabs}
\usepackage{multicol}
\usepackage{placeins}
\usepackage{caption} 
\usepackage{subcaption} 
\usepackage{graphicx}
\usepackage{mathdots}
\usepackage{hyperref}
\usepackage[capitalise]{cleveref}
\usepackage{mathrsfs}
\usepackage{txfonts}
\usepackage{amsopn}
\usepackage{checkend}
\usepackage{longtable,geometry}
\usepackage{physics}
\usepackage{etoolbox} 
\usepackage{xcolor}
\usepackage{cases}
\usepackage{scalefnt}
\usepackage{float}
\newtheorem{theorem}{Theorem}

\newtheorem{lemma}{Lemma}

\begin{document}
	
\title[Traveling-Wave Solutions]{Traveling-Wave Solutions to the Nonlinear Double Degenerate Parabolic Equation of Turbulent Filtration with Absorption}
	%\thanks{Department of Mathematics, Florida Institute of Technology, Melbourne, FL 32901}

	\author[A. Prinkey]{Adam Prinkey}
	\address{Department of Mathematical Sciences, Florida Institute of Technology, Melbourne, FL 32901}
	\email{aprinkey2009@my.fit.edu}

\maketitle
\begin{center}
Department of Mathematical Sciences, Florida Institute of Technology, Melbourne, FL 32901
\end{center}
\begin{abstract}
In this paper we prove the existence of finite traveling-wave type solutions to the nonlinear double degenerate parabolic equation of turbulent filtration with absorption.
\end{abstract}
%%%%%%%%%%%%%%%%%%%%%%%%%%%%%%%%%%%%%%%%%%%%%%%%%%%%%%%%%%
\section{Introduction}
\label{sec:intro}

In this paper we consider the nonlinear double degenerate parabolic equation of turbulent filtration with absorption
\begin{equation}\label{EQ1}
u_t=\Big(|(u^{m})_x|^{p-1}(u^{m})_x\Big)_x-bu^{\beta}, \, x \in \mathbb{R}, \, t>0, 
\end{equation}
with $mp>1, \, (m,p>0), \, 0 < \beta < 1, \text{ and } b > 0$. The condition that $mp>1$ implies that the solutions of \eqref{EQ1} travel with a finite speed of propagation (slow diffusion case). We are interested in finding finite traveling-wave solutions to \eqref{EQ1}: $u(x,t) = \varphi(kt-x)$, where the function $\varphi$ is such that: $\varphi(z) \geq 0,\, \varphi \not\equiv 0$, $\varphi(z) \rightarrow 0^+$ as $z \rightarrow -\infty$, and $\varphi(0) = 0$. 

Equation \eqref{EQ1} admits a finite traveling-wave solution if there exists $\varphi \in \mathbb{R}^+$ that satisfies the following initial-value-problem (IVP)
\begin{equation}\label{EQ3}\begin{cases}
\Big(|(\varphi^{m})'|^{p-1}(\varphi^{m})'\Big)'- k\varphi'-b\varphi^{\beta} = 0, \\
\varphi(0) = (\varphi^m)'(0) = 0,
\end{cases}\end{equation}
where $\varphi(z) \equiv 0$ for all $z<0$. All derivatives are understood in the weak sense. 

The following is the main result of this paper.
\begin{theorem}\label{Theorem 1}
There exists a finite traveling-wave solution to \eqref{EQ1}: $\varphi(kt-x)$, with $\varphi(0) = 0$ if $k \neq 0$. Further, we have
\end{theorem}
\begin{enumerate}
\item $\lim\limits_{z \to 0^+} z^{-\frac{1+p}{mp-\beta}}\varphi(z) = \Big[\frac{b(mp-\beta)^{1+p}}{(m(1+p))^{p}p(m+\beta)}\Big]^{\frac{1}{mp-\beta}} := C_*$, if $p(m+\beta) < 1+p$,
\item $\lim\limits_{z \to +\infty} z^{-\frac{1+p}{mp-\beta}}\varphi(z) = C_*$, if $p(m+\beta) > 1+p$,
\item $\lim\limits_{z \to +\infty} z^{-\frac{p}{mp-1}}\varphi(z) = \Big(\frac{mp-1}{mp}\Big)^{\frac{p}{mp-1}}k^{\frac{1}{mp-1}}$, if $k > 0$, $p(m+\beta)<1+p$,
\item $\lim\limits_{z \to 0^+} z^{-\frac{p}{mp-1}}\varphi(z) = \Big(\frac{mp-1}{mp}\Big)^{\frac{p}{mp-1}}k^{\frac{1}{mp-1}}$, if $k > 0$, $p(m+\beta)>1+p$,
\item $\lim\limits_{z \to +\infty} z^{-\frac{1}{1-\beta}}\varphi(z) = \Big((1-\beta)\Big({-\frac{b}{k}}\Big)\Big)^{\frac{1}{1-\beta}}$, if $k < 0$, $p(m+\beta)<1+p$,
\item $\lim\limits_{z \to 0^+} z^{-\frac{1}{1-\beta}}\varphi(z) = \Big((1-\beta)\Big({-\frac{b}{k}}\Big)\Big)^{\frac{1}{1-\beta}}$, if $k < 0$, $p(m+\beta)>1+p$.
\end{enumerate}

The existence of traveling-wave solutions with interfaces for the nonlinear reaction-diffusion equation (\eqref{EQ1} with $p=1$) is pursued in \cite{HerreroVazquez1}. Existence of traveling-wave type solutions to \eqref{EQ1} for the parabolic $p$-Laplacian equation is considered in \cite{Mu1}.

It is of note that currently there is a well established general theory of nonlinear degenerate parabolic equations, beginning with \cite{Oleinik}; see also \cite{zeldovich,Barenblatt1,Vazquez2,Dibe-Sv,Abdulla6,Abdulla7,Abdulla8,Abdulla9,Abdulla11,Abdulla12,Abdulla14,Abdulla16,Abdulla18,Abdulla20,Abdulla21,Abdulla25,
Herrero_Vazquez,HerreroVazquez1,Kalashnikov1,Kalashnikov4,Kersner1,Vasquez1,Gala1,shmarev,Abdullaev1,Abdulla4,Abdullaev1,Barenblatt2,BCP,DibeHerrero2,Grundy1,GrundyPeletier,HerreroPierre,KPV,Leibenson,shmarev2015interfaces}). Boundary value problems for \eqref{EQ1} have been investigated in \cite{Kalashnikov2,Kalashnikov3,Esteban,Tsutsumi,Ishige,Degtyarev, Ivanov1, Ivanov2,Vespri1}.

Let
\begin{equation}\label{IF1}
u(x,0)=u_0(x), \, x\in \mathbb{R}.
\end{equation}
The solution of the Cauchy problem (CP) \eqref{EQ1}, \eqref{IF1} is understood in the weak sense (see Definition 1 from \cite{AbdullaPrinkey1}). The full classification of the interfaces
\[\eta(t):=\sup\{x:u(x,t)>0\}, \, \eta(0)=0,\]
and local solutions near the interfaces for the Cauchy problem \eqref{EQ1}, \eqref{IF1} is established in \cite{AbdullaPrinkey1} in the slow diffusion case ($mp>1$) and in \cite{AbdullaPrinkey2} for the fast diffusion case ($0<mp<1$). This classification is done for the nonlinear reaction-diffusion equation (\eqref{EQ1} with $p=1$) in \cite{Abdulla1} for the slow diffusion case and in \cite{Abdulla3} for the fast diffusion case; and for the parabolic $p$-Laplacian diffusion-reaction equation (\eqref{EQ1} with $m=1$) in \cite{AbdullaJeli} for the slow diffusion case and in \cite{AbdullaJeli2} for the fast diffusion case. The use of finite traveling-wave solutions was essential to prove asymtotic results for the interface and the local solution near the interface in the cases where diffusion and reaction forces are in balance.

The organization of the paper is as follows: in \cref{sec:TWS} we formulate and prove some preliminary results which are necessary for the proof of main result and in \cref{sec:main result} we prove the main result, \cref{Theorem 1}. \\
%%%%%%%%%%%%%%%%%%%%%%%%%%%%%%%%%%%%%%%%%%%%%%%%%%%%%%%%%%
\section{Traveling-Wave Solutions and Phase-Space Analysis}
\label{sec:TWS}

In this section we'll apply phase-space analysis to find finite traveling-wave solutions for \eqref{EQ1}. We aim to analyze the phase portrait for problem \eqref{EQ3}. We establish an essential monotinicity property of $\varphi$.
\begin{lemma}\label{Lemma 1}
If $\varphi$ is a positive solution to \eqref{EQ3}, then $\varphi$ is increasing on (0,$+\infty$). 
\end{lemma}
\begin{proof}[Proof of \cref{Lemma 1}] 
If $k < 0$, the result easily follows since the solution to\eqref{EQ3} cannot obtain a local maximum. For $k>0$, the result follows as in the analogous proof for the $p$-Laplacian equation in \cite{Mu1} by choosing 
\[\Phi(z) := \frac{p}{p+1}|(\varphi^{m})'|^{p+1}-\frac{bm}{m+\beta}\varphi^{m+\beta}.\]
\iffalse
For $k < 0$, suppose that $\varphi$ is \emph{not} increasing in $\mathbb{R}^+$, then $\exists \, z_0$ such that $\varphi$ is strictly increasing on $(0, z_0)$ and $z_0$ is a local maximum, then 
\[\Big(|(\varphi^{m})'|^{p-1}(\varphi^{m})'\Big)'(z_0) \leq 0,\]
but, clearly, by \eqref{EQ3}
\[\Big(|(\varphi^{m})'|^{p-1}(\varphi^{m})'\Big)'(z_0) > 0,\]
which is a contradiction, so it follows that: $\varphi'(z) > 0$, $\forall\, z > 0$. \\
For $k > 0$ we define a function
\[\Phi(z) := \frac{p}{p+1}|(\varphi^{m})'|^{p+1}-\frac{bm}{m+\beta}\varphi^{m+\beta},\]
so we have
\[\Phi'(z) = p|(\varphi^{m})'|^{p-1}(\varphi^m)''(\varphi^m)'-bm\varphi^{m+\beta-1}\varphi',\]
which, by \eqref{EQ3}, can be written as
\[((|(\varphi^{m})'|^{p-1}(\varphi^{m})')'-b\varphi^{\beta})(\varphi^m)' = k\varphi'(\varphi^m)' = \frac{k}{m}\varphi^{1-m}[(\varphi^m)']^2 \geq 0.\]
If we assume $\varphi$ is not increasing on (0,$+\infty$), letting $z_0$ denote the first zero of $\varphi'$, we have
\[0 = \Phi(0) \leq \Phi(z_0) = -\frac{bm}{m+\beta}\varphi^{m+\beta}(z_0) < 0,\]
which is a contradiction, so it follows that: $\varphi'(z) > 0$, $\forall \, z > 0$. Lemma \ref{Lemma 1} is proved.
\fi
\end{proof}

Now, we want to show that there exists such a $\varphi(z) > 0$. We introduce the following change of variable
\[ \Theta = \varphi \text{ and } \Upsilon = ((\varphi^m)')^p,\]
it follows that
\[\Theta' = \frac{1}{m}\Theta^{1-m}\Upsilon^{\frac{1}{p}} \text{ and } \Upsilon' = b\Theta^\beta+\frac{k}{m}\Theta^{1-m}\Upsilon^{\frac{1}{p}},\]
where $(\Theta, \Upsilon)$ starts from (0,0) at $z = 0$, exists for any  $z \in \mathbb{R}^+$, and are contained in the first quadrant: $Q_1 = \{(\Theta,\Upsilon): \Theta, \Upsilon > 0\}$ for $z > 0$. We claim that there exists a unique solution, or trajectory, $\Upsilon(\Theta)$. 
Consider
\begin{equation}\label{PP1}\begin{cases}
\frac{d\Upsilon}{d\Theta} = f(\Theta,\Upsilon) = k + bm\Theta^{m+\beta-1}\Upsilon^{-\frac{1}{p}}, \\
\Upsilon(0) = 0.
\end{cases}\end{equation}

As done in \cite{Mu1} for the analogous problem for the $p$-Laplacian equation, we find the nontrivial trajectories, $\Upsilon(\Theta)$, to \eqref{PP1}, in two steps. First we prove the global existence of the solution of the following perturbed IVP
\begin{equation}\label{PP2}\begin{cases}
\frac{d\Upsilon}{d\Theta} = f(\Theta,\Upsilon) = k + bm\Theta^{m+\beta-1}\Upsilon^{-\frac{1}{p}}, \\
\Upsilon(0) = \varepsilon, \, \varepsilon > 0.
\end{cases}\end{equation}
Since $f(\Theta,\Upsilon)$ is locally Lipshitz continuous in $\mathbb{R}^+ \times (\varepsilon, +\infty)$, there exists a unique local solution to \eqref{PP2}, $\Upsilon_\varepsilon$. For $k>0$ and for $k<0$ with $p(m+\beta)>1+p$, the proof of the existence of a global solution to \eqref{PP2} follows as in the proof of the existence of a global solution to the analgous IVP for the $p$-Laplacian equation in \cite{Mu1}. \\
 $f(\Theta,\Upsilon)$ is strictly increasing and satisfies the following inequality
\[\frac{d\Upsilon_\varepsilon}{d\Theta} \leq k + bm\Theta^{m+\beta-1}\varepsilon^{-\frac{1}{p}},\]
so it follows that
\[\Upsilon_\varepsilon \leq k\Theta + \frac{bm}{m+\beta}\Theta^{m+\beta}\varepsilon^{-\frac{1}{p}} + \varepsilon,\]
hence, $\Upsilon_\varepsilon$ is a global solution. 
Let $k>0$ and $p(m+\beta)>1+p$.
Let $p(m+\beta)>1+p$. For $k < 0$, define the curve 
\[\widetilde{C} : \Upsilon(\Theta) = \bigg(-\frac{k}{bm}\Theta^{1-m-\beta}\bigg)^{-p},\]
then we have $f(\Theta,\Upsilon) = 0$ on $\widetilde{C}$ and $\widetilde{C}$ divides the first quadrant, $\Omega_1$, into two regions: $R_l = \{(\Theta,\Upsilon): f(\Theta,\Upsilon) < 0\}$ and $R_r = \{(\Theta,\Upsilon): f(\Theta,\Upsilon)>0\}$, see \cref{fig:1}. $\Upsilon_\varepsilon$ starts in region $R_l$, then $\Upsilon_\varepsilon$ must cross $\widetilde{C}$ at some point with horizontal tangent and after $\Upsilon_\varepsilon$ lies in the region $R_r$, where $\Upsilon_\varepsilon$ is strictly increasing. Hence there exists $\delta_\varepsilon>0$ such that $\Upsilon_\varepsilon$ attains its minimum, $M_\varepsilon$: $\Upsilon_\varepsilon(\delta_\varepsilon) = M_\varepsilon$, which lies on $\widetilde{C}$ and is strictly positive. So we have
\[\frac{d\Upsilon_\varepsilon}{d\Theta} \leq k + bm\Theta^{m+\beta-1}M_\varepsilon^{-\frac{1}{p}},\]
so it follows that
\[\Upsilon_\varepsilon(\Theta) \leq k(\Theta-\delta_\varepsilon) + \frac{bm}{m+\beta}(\Theta^{m+\beta}-\delta_{\varepsilon}^{m+\beta})M_\varepsilon^{-\frac{1}{p}} + M_\varepsilon.\]
Let $k<0$ with $p(m+\beta)<1+p$. The difference from the previous case is that
\[\widetilde{C}: (0,+\infty)\to(+\infty,0),\]
see \cref{fig:2}. Since
\[\frac{d\Upsilon_\varepsilon}{d\Theta} > 0 \text { if } \Upsilon_{\varepsilon} < \bigg(-\frac{k}{bm}\Theta^{1-m-\beta}\bigg)^{-p},\] 
$\Upsilon_\varepsilon$ is increasing to the left of $\widetilde{C}$. Then $\Upsilon_{\varepsilon}$ must cross $\widetilde{C}$ with horizontal tangent, after that $\Upsilon_\varepsilon$ will be strictly decreasing. It follows that $\Upsilon_\varepsilon$ is a global solution to \eqref{PP2} if $k<0$.\\
Next we prove the global existence of the CP
\begin{equation}\label{PP3}\begin{cases}
\frac{d\Upsilon}{d\Theta} = f(\Theta,\Upsilon) = k + bm\Theta^{m+\beta-1}\Upsilon^{-\frac{1}{p}}, \\
\Upsilon(\varepsilon) = 0, \, \varepsilon > 0.
\end{cases}\end{equation}
To do this, we consider the following CP for the inverse function of $\Upsilon$, denoted as $v$
\begin{equation}\label{CP1}\begin{cases}
\frac{dv}{dt} = g(v,t) = \frac{1}{f(v,t)} = \frac{t^\frac{1}{p}}{kt^\frac{1}{p} + bmv^{m+\beta-1}}, \\
v(0) = \varepsilon, \, \varepsilon > 0. 
\end{cases}\end{equation}
Since the right hand side of \eqref{CP1} is Lipshitz continuous, there exists a local solution, $v_\varepsilon$, to the CP \eqref{CP1}. For $k>0$ and for $k<0$ with $p(m+\beta)>1+p$, as for \eqref{PP2},  the proof of the existence of a global solution to \eqref{PP3} follows as in the proof of the existence of a global solution to the analgous IVP for the $p$-Laplacian equation in \cite{Mu1}.
We have the following inequality
\[0 \leq \frac{dv_\varepsilon}{dt} \leq \frac{1}{k},\]
it follows that $v_\varepsilon$ is a global solution to the CP. Let $p(m+\beta)>1+p$. For $k < 0$ we denote $\widetilde{C}$ as the curve where $f(v_\varepsilon,t) = 0$. Then, as before, $\widetilde{C}$ divides $\Omega_1$ into two regions: $R_l = \{(v,t): f(v,t) > 0\}$ and $R_r = \{(v,t): f(v,t) < 0\}$, see \cref{fig:3}. $v_\varepsilon$ starts in region $R_l$ and $\frac{dv_\varepsilon}{dt}$ is strictly positive and tends to $+\infty$ as $f(v_\varepsilon,t) \rightarrow 0^+$. It follows that $v_\varepsilon$ is strictly increasing and never touches $\widetilde{C}$. Therefore, $v_\varepsilon$ is a global solution to the CP. Moreover, we have that
\[\lim_{t \to +\infty} v_\varepsilon(t) = +\infty.\]
Hence, $v_\varepsilon$ is one-to-one from $[0,+\infty)$ to $[\varepsilon, +\infty)$. Now, let $w_\varepsilon$ denote the inverse function of $v_\varepsilon$, defined from $[\varepsilon, +\infty)$ to $[0,+\infty)$. Clearly, $w_\varepsilon$ satisfies the following CP
\begin{equation}\label{CP2}\begin{cases}
\frac{dw_\varepsilon}{d\Theta} = f(\Theta,w) = k + bm\Theta^{m+\beta-1}w^{-\frac{1}{p}}, \\
w_\varepsilon(\varepsilon) = 0, \, \varepsilon > 0.
\end{cases}\end{equation}
Therefore, the CP \eqref{PP3} has a unique global solution for any $\varepsilon > 0$. Now, let $k<0$ with $p(m+\beta)<1+p$.  As before, we define the curve where $f(v,t)=0$ by $\widetilde{C}$. We denote the region to the left of $\widetilde{C}$ as $R_l = \{(v,t): f(v,t) > 0\}$ and to the region to the right of $\widetilde{C}$ as $R_r = \{(v,t): f(v,t) < 0\}$, see \cref{fig:4}. Since $v$ is increasing in $R_l$ it must cross $\widetilde{C}$ with vertical tangent, however, this is impossible. Let $t_\varepsilon$ be such that $v(t_\varepsilon) = M_\varepsilon \, \in \widetilde{C}$. Consider the function $w$ such that
\[w: [\varepsilon, M_\varepsilon] \to [0, t_\varepsilon]. \]
Then $w$ is the inverse function of $v$ in $[0, t_\varepsilon]$ and so solves the following problem
\begin{equation}\label{WINVERSE}\begin{cases}
\frac{dw}{dt} =  k + bmw^{m+\beta-1}t^{-\frac{1}{p}} = f (w,t), \\
w(\varepsilon) = 0, \, w(M_\varepsilon) = t_\varepsilon, \, \varepsilon > 0.
\end{cases}\end{equation}
Let $\widehat{C}$ denote the curve where $f(w,t)=0$. So $w$ enters the region to the right of $\widehat{C}$ with horizontal tangent and since if $t > M_\varepsilon$, then $w(t)$ is decreasing, we have that $w$ cannot cross $\widehat{C}$ again since it must cross with horizontal tangent, which is a contradiction. It follows that the solution, $w$, to problem \eqref{WINVERSE} is global and so there exists a global solution to problem \eqref{PP3} if $k<0$.

\begin{lemma}\label{Lemma 2}
The problem \eqref{PP1} has a unique global solution. 
\end{lemma}
The proof of Lemma \ref{Lemma 2} follows as in the proof of existence and uniqueness of solution for the analogous problem for the $p$-Laplacian equation in \cite{Mu1}.
\newpage
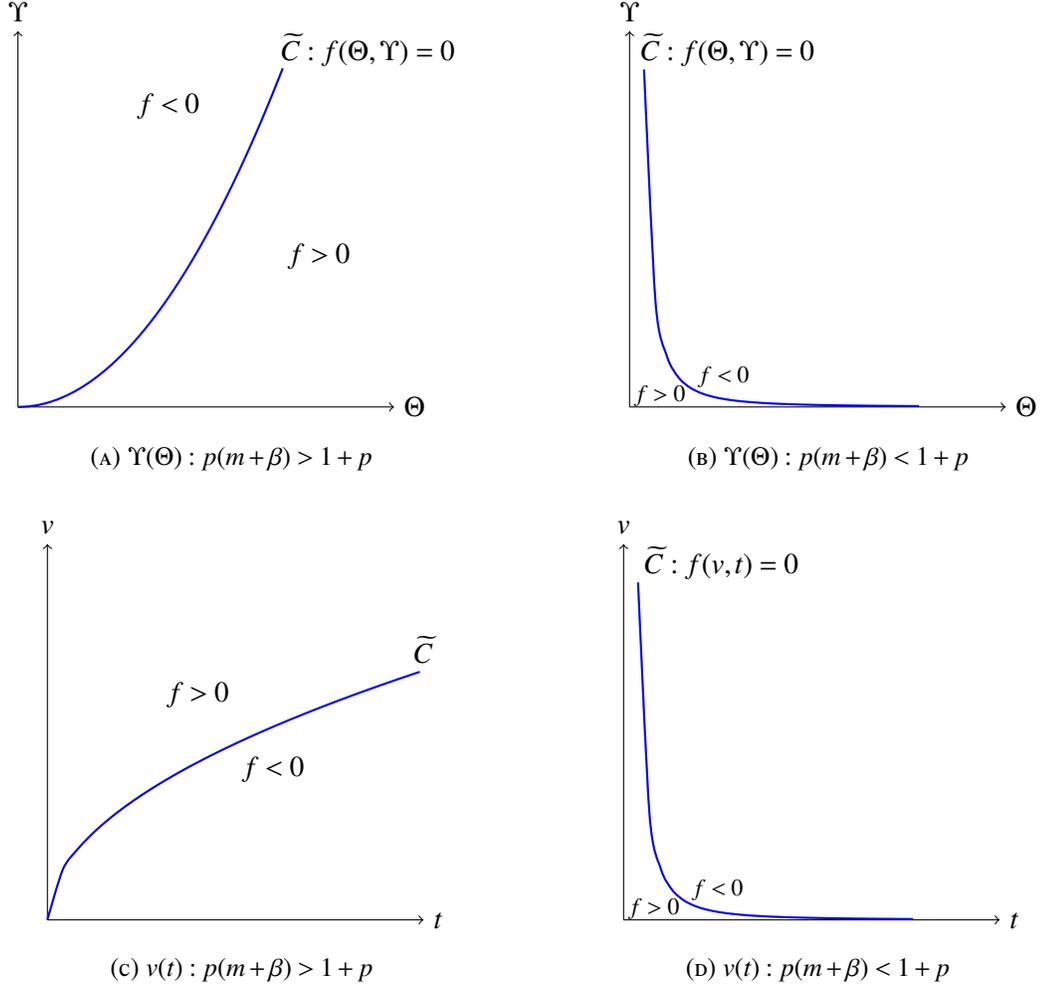
\begin{figure}[H]
        \centering
        \begin{subfigure}[b]{0.475\textwidth}
            \centering
\begin{tikzpicture}
   \draw[->] (0.,0) -- (5,0) node[right] {$\Theta$};
   \draw[->] (0,0) -- (0,5) node[above] {$\Upsilon$};
   \draw[scale=0.55,domain=0.0:6.4,smooth, variable=\x,blue, thick] plot ({\x},{(0.2)*((\x)*(\x))});
\node at (4.65,4.75) {$\widetilde{C}: f(\Theta,\Upsilon)=0$};
\node at (4, 2) {$f>0$};
\node at (2,4) {{$f<0$}};
\end{tikzpicture}
\caption[1]%
                       {{\small $\Upsilon(\Theta): p(m+\beta)>1+p$}}    
            \label{fig:1}
        \end{subfigure}
        \hfill
        \begin{subfigure}[b]{0.475\textwidth}  
            \centering 
\begin{tikzpicture}
   \draw[->] (0.,0) -- (5,0) node[right] {$\Theta$};
   \draw[->] (0,0) -- (0,5) node[above] {$\Upsilon$};
   \draw[scale=0.55,domain=0.35:7,smooth, variable=\x,blue, thick] plot ({\x},{1/((\x)*(\x))});
\node at (1.30,4.75) {$\widetilde{C}: f(\Theta,\Upsilon)=0$};
\node at (1.25, 0.40) {\footnotesize$f<0$};
\node at (0.40,0.15) {\footnotesize{$f>0$}};
\end{tikzpicture}
\caption[2]%
                        {{\small $\Upsilon(\Theta): p(m+\beta)<1+p$}}    
            \label{fig:2}
        \end{subfigure}
        \vskip\baselineskip
        \begin{subfigure}[b]{0.475\textwidth}   
            \centering 
\begin{tikzpicture}
   \draw[->] (0.,0) -- (5,0) node[right] {$t$};
   \draw[->] (0,0) -- (0,5) node[above] {$v$};
   \draw[scale=0.55,domain=0.0:9,smooth, variable=\x,blue, thick] plot (\x,{sqrt(4*\x)});
\node at (5,3.6) {$\widetilde{C}$};
\node at (3, 2) {$f<0$};
\node at (2,3) {{$f>0$}};
\end{tikzpicture}
\caption[3]%
                       {{\small $v(t): p(m+\beta)>1+p$}}    
            \label{fig:3}
        \end{subfigure}
        \quad
        \begin{subfigure}[b]{0.475\textwidth}   
            \centering 
\begin{tikzpicture}
   \draw[->] (0.,0) -- (5,0) node[right] {$t$};
   \draw[->] (0,0) -- (0,5) node[above] {$v$};
   \draw[scale=0.55,domain=0.35:7,smooth, variable=\x,blue, thick] plot ({\x},{1/((\x)*(\x))});
\node at (1.30,4.75) {$\widetilde{C}: f(v,t)=0$};
\node at (1.25, 0.40) {\footnotesize$f<0$};
\node at (0.40,0.15) {\footnotesize{$f>0$}};
\end{tikzpicture}
\caption[4]%
                       {{\small $v(t): p(m+\beta)<1+p$}}    
            \label{fig:4}
        \end{subfigure}
\caption[ The average and standard deviation of critical parameters ]
               {\small Trajectories $\Upsilon(\Theta)$ and $v(t)$} 
        \label{fig:trajectories}
    \end{figure}

Let $\Upsilon = ((\varphi^m)')^p$ be a solution of the problem \eqref{PP1}. For the problem
\begin{equation}\label{EQ5}
\frac{d\varphi}{dz} =\frac{1}{m}(\varphi(z))^{1-m} \Upsilon^{\frac{1}{p}}(\varphi(z)), \, \varphi(0) = 0,
\end{equation}

there exists a unique maximal solution defined on $(-\infty, \varrho)$ such that
\[\lim_{z \to \varrho^-} \varphi(z) = +\infty.\]
By \eqref{EQ5} we have that $(\varphi^m)'(0) = \Upsilon^{\frac{1}{p}}(0) = 0$,
so we can continue $\varphi$ by zero on $(-\infty,0)$. On the other side, $\varphi$ is strictly increasing, and
\[\lim_{z \to \varrho^-} \varphi(z) = +\infty,\]
if $\varrho$ is finite. By \eqref{EQ5} and the boundedness of $\Upsilon^{-\frac{1}{p}}$, the above limit also holds if $\varrho = +\infty$.

The solution of \eqref{EQ5} defined on $(-\infty,\varrho)$ satisfies

\begin{equation}\label{EQ6}\begin{cases}
\big(|(\varphi^{m})'|^{p-1}(\varphi^{m})'\big)'- k\varphi'-b\varphi^{\beta} = 0, \text{ on } (-\infty, \varrho),  \\
\varphi(0) = (\varphi^m)'(0) = 0.
\end{cases}\end{equation}

The solution to \eqref{EQ6} is global. To prove it, we will need the following result.\\

\begin{lemma}\label{Lemma 3}
Let $\Upsilon$ be a solution of the problem \eqref{PP1}, then
\end{lemma}
\begin{enumerate}
   \item $\Upsilon(\Theta) \sim \Big[\frac{bm(1+p)}{p(m+\beta)}\Big]^{\frac{p}{1+p}}\Theta^{\frac{p(m+\beta)}{1+p}}$, as $\Theta \to 0^+$, if $p(m+\beta) < 1+ p$,
   \item $\Upsilon(\Theta) \sim \Big[\frac{bm(1+p)}{p(m+\beta)}\Big]^{\frac{p}{1+p}}\Theta^{\frac{p(m+\beta)}{1+p}}$, as $\Theta \to +\infty$, if $p(m+\beta) > 1+ p$,
   \item $\Upsilon(\Theta) \sim k\Theta$, as $\Theta \to +\infty$, if $k>0$, $p(m+\beta) < 1+p$,
   \item $\Upsilon(\Theta) \sim k\Theta$, as $\Theta \to 0^+$, if $k>0$, $p(m+\beta) > 1 + p$,
   \item $\Upsilon(\Theta) \sim \big(-\frac{k}{bm}\big)^{-p}\Theta^{p(m+\beta-1)}$, as $\Theta \to +\infty$, if $k<0$, $p(m+\beta) < 1 + p$,
   \item $\Upsilon(\Theta) \sim \big(-\frac{k}{bm}\big)^{-p}\Theta^{p(m+\beta-1)}$, as $\Theta \to 0^+$, if $k<0$, $p(m+\beta) > 1 + p$.
\end{enumerate}

\begin{proof}[Proof of \cref{Lemma 3}]
We begin by proving formulas (1) and (2). We apply nonlinear scaling as follows: we choose $\Upsilon_l(\Theta) = l\Upsilon(l^{\gamma}\Theta)$, with $l > 0$ and $\gamma$ to be determined. 
\begin{center}
$\Upsilon_l(\Theta) = l\Upsilon(l^{\gamma}\Theta) \iff \Upsilon(\Theta) = l^{-1}\Upsilon_l(l^{-\gamma}\Theta)$.
\end{center}
We set $Z = l^{\gamma}\Theta$. It follows from \eqref{PP1} that
\begin{align}
\frac{d\Upsilon_l}{d\Theta} = l^{1+\gamma}\frac{d\Upsilon}{dZ} = l^{1+\gamma}\Big(k+bmZ^{m+\beta-1}\Upsilon^{-\frac{1}{p}}\Big) \\
= kl^{1+\gamma}+bml^{1+\gamma}l^{\gamma(m+\beta-1)}l^{\frac{1}{p}}\Theta^{m+\beta-1}\Upsilon_l^{-\frac{1}{p}}. \nonumber
\end{align} 
We choose $\gamma$ such that
\begin{center}
$1+\gamma+\gamma(m+\beta-1)+\frac{1}{p} = 0 \implies \gamma = -\frac{1+p}{p(m+\beta)}$.
\end{center}
So we have that
\begin{equation}\label{Scale2}
\frac{d\Upsilon_l}{d\Theta} = kl^{\frac{p(m+\beta)-(1+p)}{p(m+\beta)}}+bm\Theta^{m+\beta-1}\Upsilon_l^{-\frac{1}{p}}.
\end{equation}
From our previous results we that there exists a unique solution to \eqref{Scale2}. To prove formula 1, since $p(m+\beta)<1+p$, we set
\begin{center}
$$\lim_{l \to +\infty} \Upsilon_l(\Theta) = \widetilde{\Upsilon}(\Theta),$$
\end{center}
where $\widetilde{\Upsilon}(\Theta)$ solves
\begin{equation}\label{Scale3}\begin{cases}
\frac{d\Upsilon}{d\Theta} = bm\Theta^{m+\beta-1}\Upsilon^{-\frac{1}{p}}, \\
\Upsilon(0) = 0.
\end{cases}\end{equation}
The existence of the above limit follows from a similar argument used to prove an analogous limit in the proof of formula (3).
The ODE in \eqref{Scale3} is separable. Separating variables and integrating we have that
\begin{equation}\label{Scale4}
\widetilde{\Upsilon}(\Theta) = \Bigg[\frac{bm(1+p)}{p(m+\beta)}\Bigg]^{\frac{p}{1+p}}\Theta^{\frac{p(m+\beta)}{1+p}}.
\end{equation}
Recall that $Z = l^{\gamma}\Theta \implies \Theta = l^{-\gamma}Z$. So we have that
\begin{center}
$\Theta^{\frac{p(m+\beta)}{1+p}} = l^{-\frac{\gamma p(m+\beta)}{1+p}}Z^{\frac{p(m+\beta)}{1+p}} = lZ^{\frac{p(m+\beta)}{1+p}}$.
\end{center}
It follows that
\begin{align}
\lim_{l \to +\infty} \Upsilon_l(\Theta) = \lim_{l \to +\infty} l\Upsilon(Z) = \widetilde{\Upsilon}(\Theta)  \nonumber\\
= \Bigg[\frac{bm(1+p)}{p(m+\beta)}\Bigg]^{\frac{p}{1+p}}lZ^{\frac{p(m+\beta)}{1+p}} \nonumber\\
\implies \lim_{Z \to 0^+} \frac{\Upsilon(Z)}{Z^{\frac{p(m+\beta)}{1+p}}} = \Bigg[\frac{bm(1+p)}{p(m+\beta)}\Bigg]^{\frac{p}{1+p}}. \nonumber
\end{align}
Therefore,
\begin{center}
$\Upsilon(\Theta) \sim \Big[\frac{bm(1+p)}{p(m+\beta)}\Big]^{\frac{p}{1+p}}\Theta^{\frac{p(m+\beta)}{1+p}}, \text { as } \Theta \to 0^+$.
\end{center}
Note that formula (2), where $p(m+\beta) >1+p$, follows from the same procedure by setting
\begin{center}
$$\lim_{l \to 0^+} \Upsilon_l(\Theta) = \widetilde{\Upsilon}(\Theta).$$
\end{center}
To prove formulas (3) and (4) we let $k>0$ and proceed as in the proof of formulas (1) and (2). We choose the same scale as follows
\begin{center}
$\Upsilon_l(\Theta) = l\Upsilon(l^{\gamma}\Theta) \iff \Upsilon(\Theta) = l^{-1}\Upsilon_l(l^{-\gamma}\Theta)$.
\end{center}
We set $Z = l^{\gamma}\Theta$. It follows from \eqref{PP1} that
\begin{align}
\frac{d\Upsilon_l}{d\Theta} = l^{1+\gamma}\frac{d\Upsilon}{dZ} = l^{1+\gamma}\Big(k+bmZ^{m+\beta-1}\Upsilon^{-\frac{1}{p}}\Big) \\
= kl^{1+\gamma}+bml^{1+\gamma}l^{\gamma(m+\beta-1)}l^{\frac{1}{p}}\Theta^{m+\beta-1}\Upsilon_l^{-\frac{1}{p}}. \nonumber
\end{align} 
Now, we choose $\gamma$ such that
\begin{center}
$1+\gamma = 0 \implies \gamma = -1$.
\end{center}
So we have that
\begin{equation}\label{Scale5}
\frac{d\Upsilon_l}{d\Theta} = k+bml^{\frac{1+p-p(m+\beta)}{p}}\Theta^{m+\beta-1}\Upsilon_l^{-\frac{1}{p}}.
\end{equation}
From our previous results we that there exists a unique solution to \eqref{Scale5}. To prove formula (3), since  $p(m+\beta) < 1+p$, we set
\begin{center}
$$\lim_{l \to 0^+} \Upsilon_l(\Theta) = \widetilde{\Upsilon}(\Theta).$$
\end{center}
To prove the existence of this limit, let $0 \leq \Gamma < \Delta < +\infty$. We show
\begin{enumerate}
\item $\{\Upsilon_l\}$ is uniformly bounded, i.e., $\abs{\Upsilon_l(\Theta)} \leq C$, for all $\Theta \in [\Gamma, \Delta]$ and $l$, where $C$ is independent of $l$. 
\item $\{\Upsilon_l\}$ is equicontinuous, i.e., for any $\varepsilon > 0$, there exists $\delta = \delta_\varepsilon > 0$ such that for all $ \Theta, \Theta_0 \in [\Gamma, \Delta]$ we have
\[\abs{\Theta-\Theta_0} < \delta \implies \abs{\Upsilon_l(\Theta)-\Upsilon_l(\Theta_0)} < \varepsilon, \forall \, l.\]
\end{enumerate}
First we prove that $\{\Upsilon_l\}$ is uniformly bounded. Since we want to pass $l$ to zero, we fix $l \in (0,1]$. So we have that
\[\frac{d\Upsilon_l}{d\Theta} = k+bml^{\frac{1+p-p(m+\beta)}{p}}\Theta^{m+\beta-1}\Upsilon_l^{-\frac{1}{p}} \leq k+bm\Theta^{m+\beta-1}\Upsilon_1^{-\frac{1}{p}} = \frac{d\Upsilon_1}{d\Theta}.\]
Choosing $\Gamma = 0$ we have that $\Upsilon_l(0) = \Upsilon_1(0) = 0$, so by applying the comparison theorem we have
\[ 0 \leq \Upsilon_l(\Theta) \leq \Upsilon_1(\Theta), \, \forall \, \Theta \in [0,\Delta], \, \forall \, l \in (0,1].\]
It remains to show that $\frac{d\Upsilon_l}{d\Theta}$ is uniformly bounded. Let $\Theta \in [\Gamma,\Delta]$. Since $k>0$ we have that
\[\frac{d\Upsilon_l}{d\Theta} \geq k \implies \Upsilon_l(\Theta) \geq k\Theta \implies \Upsilon_l(\Gamma) \geq k\Gamma > 0 \implies \Upsilon_l^{-\frac{1}{p}}(\Gamma) \leq (k\Gamma)^{-\frac{1}{p}}. \]
So we have
\[\frac{d\Upsilon_l}{d\Theta} = k+bml^{\frac{1+p-p(m+\beta)}{p}}\Theta^{m+\beta-1}\Upsilon_l^{-\frac{1}{p}}(\Theta) \leq k+bm\Delta ^{m+\beta-1}(k\Gamma)^{-\frac{1}{p}} < +\infty.\]
This holds for all $l \in (0,1]$.
Since $\frac{d\Upsilon_l}{d\Theta}$ is uniformly bounded on $[\Gamma, \Delta]$ it follows that $\Upsilon_l(\Theta)$ is uniformly bounded on $[\Gamma, \Delta]$. 
Now we need to show that $\{\Upsilon_l\}$ is equicontinuous on $[\Gamma, \Delta]$. Let $\Theta, \Theta_0 \in [\Gamma, \Delta]$. We need to show that for any $\varepsilon > 0$, there exists $\delta = \delta_\varepsilon > 0$ such that
\[\abs{\Theta-\Theta_0} < \delta \implies \abs{\Upsilon_l(\Theta)-\Upsilon_l(\Theta_0)} < \varepsilon, \forall \, l.\]
By Lagrange's mean value theorem, for all $\theta \in [0,1]$, we have
\[\abs{\Upsilon_l(\Theta)-\Upsilon_l(\Theta_0)} = \abs{\frac{d\Upsilon_l(\Theta_0+\theta(\Theta-\Theta_0))}{d\Theta}(\Theta-\Theta_0)} \leq C\abs{\Theta-\Theta_0} < C\delta.\] Choosing $\delta = \frac{\varepsilon}{C}$ ensures that $\abs{\Upsilon_l(\Theta)-\Upsilon_l(\Theta_0)} < \varepsilon, \forall \, l$. So $\{\Upsilon_l\}$ is equicontinuous on $[\Gamma, \Delta]$.
Since $\{\Upsilon_l\}$ is both uniformly bounded and equicontinuous on $[\Gamma, \Delta]$, and since $[\Gamma, \Delta]$ is an arbitrary compact subset of $[0,+\infty)$, there exists $\widetilde{\Upsilon}(\Theta)$ such that for some subsequence $l'$ we have
\[\lim_{l'\rightarrow 0^+}\Upsilon_{l'}(\Theta)= \widetilde{\Upsilon}(\Theta), \, \, \forall \, \Theta>0.\]
Where $\widetilde{\Upsilon}(\Theta)$ solves
\begin{equation}\label{10}\begin{cases}
\frac{d{\Upsilon}}{d\Theta} = k,\, \Theta > 0,\\[.1cm]
\Upsilon(0)=0.
\end{cases}\end{equation}
So $\widetilde{\Upsilon}(\Theta) = k\Theta$, and we have
\[\lim_{l\rightarrow 0^+}\Upsilon_l(\Theta) = \lim_{l\rightarrow 0^+} l\Upsilon(l^{\gamma}\Theta) = k\Theta, \, \Theta > 0.\]
Recall that $Z = l^{\gamma}\Theta \implies \Theta = l^{-\gamma}Z$. So we have that
\begin{align}
\lim_{l \to 0^+} \Upsilon_l(\Theta) = \lim_{l \to 0^+} l\Upsilon(Z) = \widetilde{\Upsilon}(\Theta) = klZ  \nonumber\\
\implies \lim_{Z \to +\infty} \frac{\Upsilon(Z)}{Z} = k. \nonumber
\end{align}
Therefore,
\begin{center}
$\Upsilon(\Theta) \sim k\Theta, \text { as } \Theta \to +\infty$.
\end{center}
Note that formula (4), where $p(m+\beta)>1+p$, follows from the same procedure by setting
\begin{center}
$$\lim_{l \to +\infty} \Upsilon_l(\Theta) = \widetilde{\Upsilon}(\Theta).$$
\end{center}
To prove formulas (5) and (6) we let $k<0$ and proceed as in the proof of the previous formulas. We choose the same scale as follows
\begin{center}
$\Upsilon_l(\Theta) = l\Upsilon(l^{\gamma}\Theta) \iff \Upsilon(\Theta) = l^{-1}\Upsilon_l(l^{-\gamma}\Theta)$.
\end{center}
We set $Z = l^{\gamma}\Theta$. It follows from \eqref{PP1} that
\begin{align}
\frac{d\Upsilon_l}{d\Theta} = l^{1+\gamma}\frac{d\Upsilon}{dZ} = l^{1+\gamma}\Big(k+bmZ^{m+\beta-1}\Upsilon^{-\frac{1}{p}}\Big) \\
= kl^{1+\gamma}+bml^{1+\gamma}l^{\gamma(m+\beta-1)}l^{\frac{1}{p}}\Theta^{m+\beta-1}\Upsilon_l^{-\frac{1}{p}}. \nonumber
\end{align} 
Now, we choose $\gamma$ such that
\begin{center}
$1+\gamma = 1+\gamma+\gamma(m+\beta-1)+\frac{1}{p} \implies \gamma = -\frac{1}{p(m+\beta-1)}$.
\end{center}
So we have that
\begin{equation}\label{Scale8}
l^{\frac{1+p-p(m+\beta)}{p(m+\beta-1)}}\frac{d\Upsilon_l}{d\Theta} = k+bm\Theta^{m+\beta-1}\Upsilon_l^{-\frac{1}{p}}.
\end{equation}
From our previous results we that there exists a unique solution to \eqref{Scale8}. To prove formula (5), since $p(m+\beta)<1+p$, we set
\begin{center}
$$\lim_{l \to 0^+} \Upsilon_l(\Theta) = \widetilde{\Upsilon}(\Theta).$$
\end{center}
As before, we have to show that the above limit exists. In this case, it's enough to prove that $\{\Upsilon_l\}$ is uniformly bounded on any compact interval, $[\Gamma, \Delta]$. From the equation we have that
\[k + bm\Theta^{m+\beta}\Upsilon_l^{-\frac{1}{p}} \geq 0 \implies 0 \leq \Upsilon_l(\Theta) \leq \bigg(-\frac{k}{bm}\bigg)^{-p}\Theta^{p(m+\beta-1)}, \, \Theta > 0.\]
It remains to show that $\frac{d\Upsilon_l}{d\Theta}$ is uniformly bounded on $[\Gamma, \Delta]$. Consider
\[ l^{\frac{1+p-p(m+\beta)}{p(m+\beta-1)}}\frac{d\Upsilon_l}{d\Theta} =k+bm\Theta^{m+\beta-1}\Upsilon_l^{-\frac{1}{p}} \implies \frac{d\Upsilon_l}{d\Theta} = l^{\frac{p(m+\beta)-(1+p)}{p(m+\beta-1)}}\bigg(k+bm\Theta^{m+\beta-1}\Upsilon_l^{-\frac{1}{p}} \bigg),\]
\[(l+1)^{\frac{1+p-p(m+\beta)}{p(m+\beta-1)}}\frac{d\Upsilon_{l+1}}{d\Theta} = k+bm\Theta^{m+\beta-1}\Upsilon_{l+1}^{-\frac{1}{p}}\]
\[\implies \frac{d\Upsilon_{l+1}}{d\Theta} = (l+1)^{\frac{p(m+\beta)-(1+p)}{p(m+\beta-1)}}\bigg(k+bm\Theta^{m+\beta-1}\Upsilon_{l+1}^{-\frac{1}{p}} \bigg) \leq l^{\frac{p(m+\beta)-(1+p)}{p(m+\beta-1)}}\bigg(k+bm\Theta^{m+\beta-1}\Upsilon_{l+1}^{-\frac{1}{p}} \bigg).\]
Define $Z(\Theta) := \Upsilon_{l+1}(\Theta)-\Upsilon_l(\Theta)$. 
By mean value theorem, for all $\theta \in [0,1]$, we have
\[\frac{dZ}{d\Theta} \leq  l^{\frac{p(m+\beta)-(1+p)}{p(m+\beta-1)}}bm\Theta^{m+\beta-1}\bigg(\Upsilon_{l+1}^{-\frac{1}{p}}-\Upsilon_l^{-\frac{1}{p}}\bigg) =\] 
\[= - l^{\frac{p(m+\beta)-(1+p)}{p(m+\beta-1)}}\frac{bm}{p}\Theta^{m+\beta-1}(\Upsilon_l+\theta(\Upsilon_{l+1}-\Upsilon_l))^{-\frac{1+p}{p}}Z\]
\[\implies l^{\frac{1+p-p(m+\beta}{p(m+\beta-1)}}\frac{dZ}{d\Theta} \leq -\frac{bm}{p}\Theta^{m+\beta-1}(\Upsilon_l+\theta(\Upsilon_{l+1}-\Upsilon_l))^{-\frac{1+p}{p}}Z.\]
Since $Z(0)=0$, it follows from the comparison theorem that $\Upsilon_{l+1}(\Theta)\leq \Upsilon_l(\Theta), \, \Theta \in [\Gamma, \Delta]$. Hence $\{\Upsilon_l\}$ is a monotonically decreasing sequence as $l \to 0^+$, and since $\Upsilon_l(\Theta) > 0$, for all $ \Theta > 0$, there exists $\widetilde{\Upsilon}(\Theta)$ such that
\[\lim_{l \to 0^+} \Upsilon_l(\Theta) = \widetilde{\Upsilon}(\Theta).\]
Now, for any $\nu \in C^{\infty}_0(\Gamma, \Delta)$, we appeal to the integral identity
\[\displaystyle \int_{\Gamma}^{\Delta}l^{\frac{1+p-p(m+\beta)}{p(m+\beta-1)}}\Upsilon_l\nu' + \big(k+bm\Theta^{m+\beta-1}\Upsilon_l^{-\frac{1}{p}}\big)\nu d\Theta = 0.\]
Letting $l \to 0^+$ we have
\[\displaystyle \int_{\Gamma}^{\Delta} \big(k+bm\Theta^{m+\beta-1}\widetilde{\Upsilon}^{-\frac{1}{p}}\big)\nu d\Theta = 0.\]
Since $\nu$ is arbitrary we necessarily have that
\[k+bm\Theta^{m+\beta-1}\widetilde{\Upsilon}^{-\frac{1}{p}}=0.\]
Solving for $\widetilde{\Upsilon}$ we have that
\begin{equation}\label{Scale10}
\widetilde{\Upsilon}(\Theta) = \bigg(-\frac{k}{bm}\bigg)^{-p}\Theta^{p(m+\beta-1)}. 
\end{equation}
Recall that $Z = l^{\gamma}\Theta \implies \Theta = l^{-\gamma}Z$. So we have that
\begin{align}
\lim_{l \to 0^+} \Upsilon_l(\Theta) = \lim_{l \to 0^+} l\Upsilon(Z) = \widetilde{\Upsilon}(\Theta) = \bigg(-\frac{k}{bm}\bigg)^{-p}lZ^{p(m+\beta-1)}    \nonumber\\
\implies \lim_{Z \to +\infty} \frac{\Upsilon(Z)}{Z^{p(m+\beta-1)}} = \bigg(-\frac{k}{bm}\bigg)^{-p}. \nonumber
\end{align}
Therefore,
\begin{center}
$\Upsilon(\Theta) \sim \big(-\frac{k}{bm}\big)^{-p}\Theta^{p(m+\beta-1)}, \text { as } \Theta \to +\infty$.
\end{center}
The proof of formula (6) follows from a similar argument.
\end{proof}
%%%%%%%%%%%%%%%%%%%%%%%%%%%%%%%%%%%%%%%%%%%%%%%%%%%%%%%%%%
\section{Proof of the Main Result}
\label{sec:main result}

Using the results above, we prove \cref{Theorem 1}. 

\begin{proof}[Proof of \cref{Theorem 1}]
As long as $\varphi(z) \neq 0$ ($\Upsilon(\varphi(z)) \neq 0$), we can rewrite \eqref{EQ5} in the following way
\begin{equation}\label{EQ7}
m\varphi^{m-1}\Upsilon^{-\frac{1}{p}}(\varphi(z))d\varphi(z) = dz.
\end{equation}
We will prove formula (2), the proof of formula (1) and formulas (3)-(6) follows in a similar way by choosing the appropriate asymptotic formula for $\Upsilon(\Theta)$ from \cref{Lemma 3}. \\
Since $p(m+\beta) > 1+p$, from \cref{Lemma 3} we know that
\[\Upsilon(\Theta) \sim \bigg[\frac{bm(1+p)}{p(m+\beta)}\bigg]^{\frac{p}{1+p}}\Theta^{\frac{p(m+\beta)}{1+p}}, \text { as } \Theta \rightarrow +\infty.\]
By \eqref{EQ7}:
\begin{equation}\label{INT1}
m \displaystyle \int_{0}^{\varphi(z)}\Theta^{m-1}\Upsilon^{-\frac{1}{p}}(\Theta)d\Theta = z.
\end{equation}
Using this fact and using the estimate above, $\forall\,\varepsilon > 0$ we have
\begin{center}
$\Big(\frac{m(1+p)}{mp-\beta}\Big(\Big[\frac{bm(1+p)}{p(m+\beta)}\Big]^{\frac{p}{1+p}} - \varepsilon \Big)^{-\frac{1}{p}}\Big)^{-\frac{1+p}{mp-\beta}}\leq z^{-\frac{1+p}{mp-\beta}}\varphi(z) \leq \Big(\frac{m(1+p)}{mp-\beta}\Big(\Big[\frac{bm(1+p)}{p(m+\beta)}\Big]^{\frac{p}{1+p}} + \varepsilon \Big)^{-\frac{1}{p}}\Big)^{-\frac{1+p}{mp-\beta}}.$
\end{center}
Passing $z \rightarrow +\infty$, we have
\begin{center}
$\Big(\frac{m(1+p)}{mp-\beta}\Big(\Big[\frac{bm(1+p)}{p(m+\beta)}\Big]^{\frac{p}{1+p}} - \varepsilon \Big)^{-\frac{1}{p}}\Big)^{-\frac{1+p}{mp-\beta}}\leq  \liminf\limits_{z\rightarrow+\infty}z^{-\frac{1+p}{mp-\beta}}\varphi(z) \leq  \limsup\limits_{z\rightarrow+\infty}z^{-\frac{1+p}{mp-\beta}}\varphi(z) \leq \Big(\frac{m(1+p)}{mp-\beta}\Big(\Big[\frac{bm(1+p)}{p(m+\beta)}\Big]^{\frac{p}{1+p}} + \varepsilon \Big)^{-\frac{1}{p}}\Big)^{-\frac{1+p}{mp-\beta}}.$
\end{center}
Now, passing $\varepsilon \to 0^+$, we have
\[\lim\limits_{z\rightarrow+\infty}z^{-\frac{1+p}{mp-\beta}}\varphi(z) = \bigg(\frac{m(1+p)}{mp-\beta}\bigg(\bigg[\frac{bm(1+p)}{p(m+\beta)}\bigg]^{\frac{p}{1+p}}\bigg)^{-\frac{1}{p}}\bigg)^{-\frac{1+p}{mp-\beta}} = C_*.\]
Formula (2) is proved.
\end{proof}

\section*{Acknowledgement}
\label{Acknowledgement}
I would like to thank my doctoral advisor, Professor Ugur G. Abdulla, for his invaluable insights into this problem and for sharing his ideas with me over many meetings.

\FloatBarrier
      	\bibliographystyle{plain}
	\bibliography{references}
\end{document}